\documentclass[reqno]{amsart}
\usepackage[utf8]{inputenc}
\usepackage{amssymb}
\usepackage{tikz}
\usepackage{euscript}
\usepackage{extarrows}
\usepackage{mathrsfs}
\usepackage{a4wide}
\usepackage{setspace}
\usepackage{graphicx}

\makeatletter
\@namedef{subjclassname@2020}{%
\textup{2020} Mathematics Subject Classification}
\makeatother

\newtheorem{thm}{Theorem}
\newtheorem{cor}[thm]{Corollary}
\newtheorem{lem}[thm]{Lemma}
\newtheorem{pro}[thm]{Proposition}
\newtheorem{opq}[thm]{Problem}

\theoremstyle{remark}
\newtheorem{rem}[thm]{Remark}

\theoremstyle{definition}
\newtheorem{exa}[thm]{Example}

\DeclareMathOperator{\lin}{lin}

\newcommand*{\cbb}{\mathbb C}

\newcommand*{\dz}[1]{{\EuScript D}(#1)}

\newcommand*{\dzn}[1]{{\EuScript D}^\infty(#1)}
\newcommand*{\ee}{\mathcal E}
\newcommand*{\ff}{\mathcal F}

\newcommand*{\hh}{\mathcal H}

\newcommand*{\is}[2]{\langle#1,#2\rangle}
\newcommand*{\jd}[1]{\EuScript N(#1)}
\newcommand*{\kk}{\mathcal K}

\newcommand*{\nbb}{\mathbb N}

\newcommand*{\ogr}[1]{\boldsymbol B(#1)}
\newcommand*{\ob}[1]{{\EuScript R}(#1)}

\newcommand*{\smalloplus}{\raise0pt\hbox{$\scriptscriptstyle \oplus$}}

\newcommand*{\zbb}{\mathbb Z}

\newcommand*{\lambdab}{\boldsymbol \lambda}
\newcommand*{\gammab}{\boldsymbol \gamma}
\newcommand*{\mub}{\boldsymbol \mu}
\newcommand*{\wlam}{W_{\boldsymbol \lambda}}
\newcommand*{\wlamk}{W_{\lambdab,k}}

\newcommand*{\nul}[1]{\mathcal{N}(#1)}

\begin{document}
\setstretch{1.2}
\title[Complex symmetric weighted shifts]{On complex symmetric weighted shifts}

\author[C. Benhida]{Chafiq Benhida}
\address{Chafiq Benhida, UFR de Mathématiques, Université des Sciences et Technologies de Lille, F-59655 Villeneuve d'Ascq CEDEX, France}
\email{chafiq.benhida@univ-lille.fr}

\author[P.\ Budzy\'{n}ski]{Piotr Budzy\'{n}ski}
\address{Piotr Budzy\'{n}ski, Katedra Zastosowa\'{n} Matematyki, Uniwersytet Rolniczy w Krakowie, ul.\ Balicka 253c, 30-198 Kra\-k\'ow, Poland}
\email{piotr.budzynski@urk.edu.pl}


\begin{abstract}
Unbounded complex symmetric weighted shifts are studied.  Complex symmetric unilateral weighted shifts whose $C^\infty$ vectors contain the image of the canonical orthonormal basis under the conjugation are shown to be decomposable into an orthogonal sum of infinitely many complex selfadjoint truncated weighted shifts, which generalizes a result of S. Zhu and C. G. Li. The bilateral case is discussed as well. Additional results, examples, and open problems are supplied.
\end{abstract}
\maketitle
\section{Introduction}
Significance of weighted shifts (classical ones and their generalizations) in operator theory is undeniable. They have been a source of many deep results and has motivated numerous studies (see \cite{nik, shi} and the bibliography therein). Due to their (formal) simplicity, they are great for verifying conjectures, constructing examples, and testing properties. They also appear in models for various classes of operators. 

Among the many properties tested in the weighted shift context was their complex symmetry. This was done by Zhu and Li in \cite{z-l-tams-2013}. They have fully characterized (classical) bounded unilateral and bilateral weighted shifts that are complex symmetric. Also, they answered the question of complex symmetry of the generalized Aluthge transform of these aforementioned operators posted in \cite{g-ieot-2008}. 

The general study of complex symmetric operator originates in the research undertaken by Garcia, Putinar, and Wogen (see \cite{g-jmaa-2007, g-otaa-2008, g-p-tams-2006, g-p-tams-2007, g-p-tmj-2008, g-w-jfa-2009, g-w-tams-2010}) but has very classical roots (for details see e.g. \cite{g-p-tams-2006}). Zhu and Li's paper is one of the many where complex symmetry of some concrete operators or ones with some additional properties (e.g., Hankel operators, truncated Toeplitz operators, normal and binormal operators, partial isometries) was studied. The literature concerning other aspects of the theory of complex symmetric operators is impressive  and still groving (see e.g., \cite{a-o-jmaa-2021, b-c-k-l-jmaa-2020, c-k-p-lma-2021, g-cm-2006, g-p-jfa-2013, g-p-p-jf-2014, g-w-ieot-2007, g-j-z-tams-2015, g-z-jot-2014, k-p-om-2015, k-k-l-lma-2020, k-l-jmaa-2016, n-s-t-jmaa-2016, z-ma-2016, z-l-j-pams-2012}).

Motivated by the research of Zhu and Li we study complex symmetry of weighted shifts in a more general context than theirs, namely we focus on unbounded weighetd shifts. The study of unbounded complex symmetric operators origins in the work of Glazman, who developed a theory similar to von Neumann's one of selfadjoint extensions of symmetric operator (see \cite{g-dansssr-1957, g-1963}; see also \cite{gal-cpam-1962, rac-jde-1985, k-jde-1981, g-p-p-jpa-2006, c-g-jmaa-2004}). We prove that the characterization of $C$-symmetry of a weighted shift $\wlam$ obtained by Zhu and Li holds whenever one assumes that the space of $C^\infty$ vectors of $\wlam$, i.e., $\dzn{\wlam}:=\bigcap_{n=1}^\infty \dz{\wlam^n}$, contains $C\ee$, where $\ee$ is the canonical basis for the underlying $\ell^2$ space (see Theorems \ref{ul} and \ref{vw01}). Along the way we show other results, among them a version of the polar decomposition for unbounded complex symmetric operators (see Theorem \ref{beck-}) and a criterion for complex selfadjointness of a complex symmetric operator (see Proposition \ref{lucek}). We also propose interesting open problems.

Last but not least, it is worth mentioning that there is an ongoing research on the subject (and other related ones) in the context of weighted shifts on directed trees. The results will be published elswhere.
\section{Preliminaries}
\subsection{Classical unilateral weighted shifts} 
Let $\nbb=\{1, 2, \ldots\}$. By $\ell^2(\nbb)$ we denote a Hilbert space all complex-valued functions on $\nbb$ that are square integrable with respect to the counting measure (endowed with a standard inner product). $\mathcal{E}$ stands  for the standard orthonormal basis of $\ell^2(\nbb)$.

Let $\lambdab=\{\lambda_n\}_{n=1}^{\infty}\subseteq \cbb$ and let $D_{\lambdab}$ be an operator in $\ell^2(\nbb)$ given by
\begin{align*}
\dz{D_{\lambdab}} &= \big\{f\in \ell^2(\nbb)\colon
\sum_{n=1}^{\infty} |f(n)\lambda_{n}|^2
<\infty\big\},\\
(D_{\lambdab} f)(n)&=\lambda_n f(n),\quad n\in\nbb,\  f\in\dz{D_{\lambdab}}.
\end{align*}
Following \cite{mlak-uiam-1988}, we define the {\em unilateral weighted shift} operator $\wlam$ with weights $\lambdab$ to be equal to the product $SD_{\lambdab}$, where $S$ is the isometric unilateral shift in $\ell^2(\nbb)$
\begin{align*}
Se_n = e_{n+1},\quad n\in \nbb.
\end{align*}
It is known that 
\begin{align*}
\dz{\wlam} &= \big\{f\in \ell^2(\nbb)\colon
\sum_{n=1}^{\infty} |f(n)\lambda_{n}|^2
<\infty\big\},\\
(\wlam f)(n)&=\lambda_{n-1} f(n-1),\quad n\in\nbb,\ f\in\dz{\wlam},
\end{align*}
with $\lambda_0=f(0):=0$, the operator $\wlam$ is closed,
$\lin \mathcal{E}$ is a core for $\wlam$ (cf.\ \cite[Eq.\ (1.7)]{mlak-uiam-1988}), and
\begin{align}\label{kdrama}
\wlam e_n = \lambda_n e_{n+1},\quad n\in
\nbb.
\end{align}
(Recall, that a linear subspace $\mathcal{G}$ of $\ell^2(\nbb)$ is called a {\em core} for a closable operator $A$ in $\ell^2(\nbb)$ if and only iff $\overline{A|_{\mathcal G}}=\overline{A}$.) The adjoint $\wlam^*$ of $\wlam$ is equal to $D_\lambda^*S^*$ and thus $\mathcal{E}\subseteq\dz{\wlam^*}$ and
\begin{align*}
\wlam^*e_n=\bar\lambda_{n-1}e_{n-1}, \quad n\in\nbb,
\end{align*}
with the usual convention: $e_0=0$.

Given $k\in\nbb$ and $\lambdab=\{\lambda_1, \ldots,\lambda_{k-1}\}\in\cbb$ we define a {\em truncated weighted shift $J_{\lambdab}\in \ogr{\cbb^{k}}$} by
\begin{align*}
J_{\lambdab} e_i=
\begin{cases}
\lambda_i e_{i+1} & \text{ for } i=1,\ldots, k-1,\\
0 & \text{ for }i=k.
\end{cases}
\end{align*}
If all $\lambda_i$'s are nonzero, we say that $J_{\lambdab}$ is of order $k-1$ (in particular, $J_{\lambdab}$ is a regular nilpotent of order $k$). 
\begin{rem}\label{cash}
If $\lambdab=\{\lambda_n\}_{n=1}^\infty$ with $\lambda_k=0$ for some $k\in\nbb$, one can easily show that $\wlam$ can be decomposed into the orthogonal sum 
\begin{align*}
\wlam=J_{\lambdab_k^-}\oplus W_{\lambdab_k^+}
\end{align*}
of a truncated weighted shift $J_{\lambdab_k^-}$ with the weights $\lambdab_k^-=\{\lambda_1, \ldots,\lambda_{k-1}\}$ and a unilateral weighted shift $W_{\lambdab_k^+}$ with the weights $\lambdab_k^+=\{\lambda_{n+k+1}\}_{n=0}^\infty$. 

Repeating the argument above, if necessary, we see that any $\wlam$ can be decomposed into the orthogonal sum 
\begin{align*}
\wlam=\bigoplus_{i\in\mathcal{I}} J_{\lambdab^{(i)}}\oplus W_{\widetilde\lambdab}
\end{align*}
of possibly infinitely many truncated weighted shifts $J_{\lambdab^{(i)}}$, $i\in\mathcal{I}$, and at most one unilateral weighted shift $W_{\widetilde \lambdab}$ with nonzero weights. Notice that the summand $W_{\widetilde \lambdab}$ is present in the above decomposition if and only if $\lambda_i=0$ for at most finite number of $i$'s.
\end{rem}
\begin{rem}\label{cacit}
Consider a unilateral weighted shift $\wlam$ and a natural number $k\geq 2$. Define functions $\phi\colon \nbb\to\nbb$ and $w\colon \nbb\to\cbb$ by
\begin{align*}
\phi(n)=\begin{cases}
0 & \text{ for } n<k,\\
n-k & \text{ for }n\geq k,
\end{cases}
\quad\quad 
w(n)=\begin{cases}
0 & \text{ for } n<k,\\
\lambda_{n-1}\lambda_{n-2}\cdots\lambda_{n-k} & \text{ for }n\geq k.
\end{cases}
\end{align*}
Then the (weighted composition) operator $W_{\lambdab,k}\colon \ell^2(\nbb)\ni\dz{W_{\lambdab,k}}\to \ell^2(\nbb)$ given by
\begin{align*}
\dz{W_{\lambdab,k}} &= \big\{f\in \ell^2(\nbb)\colon
\sum_{n=1}^{\infty} |f(\phi(n))w(n)|^2
<\infty\big\},\\
W_{\lambdab,k} f&=w (f\circ \phi),\quad f\in\dz{W_{\lambdab,k}},
\end{align*}
is well-defined and $\overline{\wlam^k}=W_{\lambdab, k}$ (see \cite[Chapter 4]{2018-bud-jab-jun-sto-lnm}). On the other hand, putting
\begin{align}\label{poezje}
\lambdab_l=\{\lambda_l, \lambda_{l+k}, \lambda_{l+2k}, \ldots\}, \quad l\in\{1, 2, \ldots, k\},
\end{align}
we may define the unilateral weighted shifts $W_{\lambdab_1}$, \ldots, $W_{\lambdab_k}$. One can see that the operators $W_{\lambdab,k}$ and $W_{\lambdab_1}\oplus \ldots\oplus W_{\lambdab_k}$ are unitarily equivalent. Also, since $\ee$ is a core for the adjoint of any unilateral weighted shift on $\ell^2(\nbb)$, we deduce that $\ee$ is a core for $\big(\overline{\wlam^k}\big)^*=\wlam^{k*}$.
\end{rem}
\subsection{Complex symmetry}
Suppose $\hh$ is a (complex separable) Hilbert space. An antilinear operator $C$ on $\hh$ satisfying $C^2=I$, where $I$ is the identity operator on $\hh$, and $\is{Cf}{Cg}=\is{g}{f}$ for all $f,g\in\hh$, is called a {\em conjugation}. A densely defined operator $T$ in $\hh$ is {\em $C$-symmetric} if $T\subseteq CT^*C$; $T$ is {\em $C$-selfadjoint} if $T= CT^*C$. Whenever we call an operator $T$  {\em complex symmetric} (resp., {\em complex selfadjoint}), we mean that $T$ is $C$-symmetric (resp., $C$-selfadjoint) with respect to some conjugation $C$ on $\hh$. 

We do not assume a complex symmetric operator to be closed. This is of course implicit in the case of complex selfadjointness. Also, in view of von Neumann's theorem, a complex symmetric operator is closable, and its closure is complex symmetric (with respect to the same conjugation as the original operator).

The following is surely well known (a proof is provided for the convenience of the reader).
\begin{lem}\label{terapia}
Let $T$ be a complex symmetric operator such that for a given $k\in\mathbb{N}$, $T^k$ is densely defined. Then the following conditions are satisfied:
\begin{itemize}
\item[(i)] $T^k$ is complex symmetric,
\item[(ii)] if  $T^k$ is complex selfadjoint, then $(T^*)^k=(T^k)^*$,
\item[(iii)] if $T$ is complex selfadjoint and $(T^*)^k=(T^k)^*$, then $T^k$ is complex selfadjoint.
\end{itemize}
\end{lem}
\begin{proof}
Suppose $T\subseteq CT^*C$ with some conjugation $C$. Since  $(T^*)^k\subseteq (T^k)^*$, we see that 
\begin{align}\label{kpop}
T^k\subseteq (CT^*C)^k=C(T^*)^k C\subseteq C(T^k)^*C.    
\end{align}
This means that $T^k$ is $C$-symmetric and so (i) holds. Using \eqref{kpop} we easily get also (ii) and (iii).
\end{proof}
\begin{rem}\label{ahmad+}
Considering a specific problem of whether $(\wlam^*)^2=(\wlam^2)^*$ we see that, since $(\wlam^*)^2\subseteq(\wlam^2)^*$, it boils down to the equality $\dz{(\wlam^*)^2}=\dz{(\wlam^2)^*}$ or equivalently to $\dz{(D_{\lambdab} \wlam)^*}=\dz{\wlam^* D_{\lambdab}^*}$. It is a matter of simple calculations to see that in general the latter equality does not hold. Indeed, one can construct a vector $f\in \dz{(D_{\lambdab}\wlam)^*}$ such that $f\notin\dz{(D_{\lambdab})^*}$ whenever the set $\Gamma=\{n\colon \lambda_{n-1}=0\}$ is infinite and $\{\lambda_n\colon n\in\Gamma\}$ is not bounded (see Example \ref{dostawa}). In particular, for such a weighted shift $\wlam$, its square $\wlam^2$ cannot be complex selfadjoint and, as noted in Remark \ref{inka}, it cannot be closed.
\end{rem}
Lemma \ref{terapia} yields the following three natural problems.
\begin{opq}
Does there exist a complex sefladjoint $T$ such that $T^2$ is not complex selfadjoint (with $T^2$ being closed)?
\end{opq}
\begin{opq}
Does there exist a complex sefladjoint $T$ such that $T^3$ is complex selfadjoint while $T^2$ is not?
\end{opq}
\begin{opq}
Does there exist a complex selfadjoint (or complex symmetric) $T$ such that the domain of $T^2$ is trivial?
\end{opq}
Let $S$ be an antilinear bounded operator on $\hh$. For any  $f\in\hh$, the mapping $g\mapsto \is{f}{Sg}$ is a linear and continuous, hence there exists a unique $h_f\in\hh$ such that $\overline{\is{h_f}{g}}=\is{f}{Sg}$ for all $g\in\hh$. Setting $S^\star f:=h_f$ we define an antilinear operator $S^\star$ on $\hh$ such that $\overline{\is{S^\star f}{g}}=\is{f}{Sg}$ for every $f,g\in\hh$. In particular, if $C$ is a conjugation, we  easily get 
\begin{align}\label{tori+}
C^\star f=C f,\quad f\in\hh.    
\end{align}
Let us note that for a given densely defined operator $B$ in $\hh$, we have
\begin{align}\label{snowstorm+}
\big(CBC\big)^*=CB^*C.
\end{align}
Indeed, observe first that $f\in\dz{(CBC)^*}$ if and only if the mapping $\dz{CBC}\ni g\mapsto \is{f}{CBCg}\in\cbb$ is bounded, equivalently, by \eqref{tori+}, the mapping $\dz{B}\ni h\mapsto \is{Cf}{Bh}\in\cbb$ is bounded. The latter holds if and only if $Cf\in\dz{B^*}$, which means that the domains of $\big(CBC\big)^*$ and $CB^*C$ are equal. Moreover, we have
\begin{align*}
\is{(CBC)^*f}{g}&=\is{f}{CBCg}=\overline{\is{C^\star f}{BCg}}=\overline{\is{B^*C f}{Cg}}\\
&=\is{C^\star B^*C f}{g}=\is{CB^*Cf}{g},\quad \forall f\in \dz{CB^*C}\textrm{ and } \forall g\in \dz{CBC}.
\end{align*}
Hence \eqref{snowstorm+} holds.

We close this section with the unbounded counterpart of \cite[Theorem 2]{g-p-tams-2007}. The proof is an adaptation of the original result.
\begin{thm}\label{beck-}
Let $T$ be $C$-selfadjoint. Suppose $T=U|T|$ is the polar decomposition of $T$. Then there exists a partial conjugation $J$ such that $U=C J$ and $J|T|=|T|J$.
\end{thm}
\begin{proof}
Define
\begin{align*}
V:=CU^*C,\quad \tilde T:=CU|T|U^*C.
\end{align*}
Since, by \eqref{snowstorm+}, $VV^*=C(U^*U)^*C$ with $U^*U$ being an orthogonal projection, we see that $(VV^*)^2=VV^*$. Clearly, $VV^*$ is selfadjoint. Thus, $V$ is a partial isometry. Since $U$ is a bounded operator on $\mathcal{H}$, $(U|T|^p)^*=|T|^pU^*$ for $p\in\{1,\frac{1}{2}\}$. This implies that $U|T|U^*=U|T|^{1/2}(U|T|^{1/2})^*$ is positive selfadjoint. Using this we easily get
\begin{align*}
\is{\tilde Tf}{f}&=\is{CU|T|U^*Cf}{C^2f}=\is{Cf}{U|T|U^*Cf}\geqslant0,\quad f\in\dz{\tilde T}
\end{align*}
and
\begin{align*}
{\tilde T}^*=C\big(U|T|U^*\big)^*C=CU|T|U^*C.
\end{align*}
Hence, $\tilde T$ is positive selfadjoint. Due to $C$-selfadjointness of $T$, we have
\begin{align}\label{alice+}
T= CT^*C=C|T|U^*C=CU^*CCU|T|U^*C=V\tilde T.
\end{align}
We prove now that the initial space of $V$ is equal to $\overline{\ob{\tilde T}}$. First, we observe that $\nul{U^*C}=\nul{U|T|U^*C}$. For this one can use the fact that $U$ is a partial isometry with the initial space equal to $\ob{T}$ which implies that for any $g\in\mathcal{H}$, $TU^*Cg=0$ if and only if $U^*Cg=0$. Thus we get
\begin{align*}
\nul{V}^\perp=\nul{CU^*C}^\perp=\nul{CU|T|U^*C}^\perp.
\end{align*}
By \eqref{alice+} and the the uniqueness of the polar decomposition (see \cite[Theorem 7.2]{sch}) we get $U=V$ and $\tilde T=|T|$.

Setting $J=U^*C$ we get $U=V=CU^*C=CJ$. Thus we have $CU=U^*C$ and $J^2=(U^*C)(CU)=P_{\overline{\ob{|T|}}}$, where $P_{\overline{\ob{|T|}}}$ is the orthogonal projection onto $\overline{\ob{|T|}}$. From $|T|=\tilde T=CU|T|U^*C$ we get $|T|=J|T|J$ which yields 
\begin{align*}
J|T|=J^2|T|J=P_{\overline{\ob{|T|}}}|T|J=|T|J.
\end{align*}
This completes the proof.
\end{proof}
The above should be compared with
\begin{thm}[\mbox{\cite[Theorem 9]{g-p-tams-2007}}]\label{resolvent}
Let $T$ be $C$-selfadjoint such that $0\in\cbb\setminus\sigma(T)$. Then there exists a conjugation $J$ such that $T=CJ|T|$ and $J$ commutes with the spectral measure of $|T|$.
\end{thm}
Clearly, Theorem \ref{resolvent} is not applicable to unilateral weighted shifts since they have $0$ in the spectrum.

Let us recall a recent paper \cite{ra-ra-na-mm-2023} where complex symmetric version of polar decomposition for bounded operators and its applications are studied. 
\section{Complex symmetry of weighted shifts}
The following generalizes \cite[Lemma 2.4]{z-l-tams-2013} and surely is of folklore-type. We provide it for future reference. We also provide its proof for the sake of completeness.
\begin{lem}
Let $(\hh_n)_{n=1}^\infty$ be a sequence of Hilbert space and $(T_n)_{n=1}^\infty$ be a sequence of densely defined operators with $T_n$ acting in $\hh_n$. Suppose that for every $n\in\nbb$ the operator $T_n$ is $C_n$-symmetric (resp., $C_n$-selfadjoint) with  $C_n\in\ogr{\hh_n}$. Then $T:=\bigoplus_{n=1}^\infty T_n$ is $C$-symmetric (resp., $C$-selfadjoint) with  $C:=\bigoplus_{n=1}^\infty C_n$.
\end{lem}
\begin{proof}
Since $T^*=\bigoplus_{n=1}^\infty T_n^*$, we see that
\begin{align*}
CT^*C=\bigoplus_{n=1}^\infty C_n T_n^*C_n\supseteq  \bigoplus_{n=1}^\infty T_n=T.
\end{align*}
For the proof of the ``$C$-selfadjoint case'' we just replace ``$\supseteq$'' by ``$=$''.
\end{proof}
It is known that a truncated weighted shift $J_{\lambdab}$ with $\lambdab=\{\lambda_1,\ldots, \lambda_{k-1}\}$ such that $|\lambda_i|=|\lambda_{k-i}|$ for $i=1,\ldots, k-1$, is complex symmetric (see \cite[Lemma 2.8]{z-l-tams-2013}). Hence we immediately get the following.
\begin{cor}\label{mymister}
Suppose $\big(J_{{\lambdab}^{(n)}}\big)_{n=1}^\infty$ is a sequence of complex symmetric truncated weighted shifts. Then $T=\bigoplus_{n=1}^\infty J_{\lambdab^{(n)}}$ is complex selfadjoint.
\end{cor}
Below we present an example of a complex symmetric weighted shift $\wlam$ such that $\wlam^k$ is complex symmetric for every $k\in\mathbb{N}$ but $\wlam^2$ is not complex selfadjoint (cf. Remark \ref{ahmad+}).
\begin{exa}\label{dostawa}
Let $(a_n)_{n=1}^\infty$ and $(b_n)_{n=1}^\infty$ be sequences of complex numbers, properties of whose will be specified later. Let $\lambdab$ be given by
\begin{align*}
    \lambda_n=\begin{cases}
0 & \text{ for } n=3k,\\
a_k & \text{ for }n=3k-2\text{ or }n=3k-1,
\end{cases}\quad (k\in\nbb).
\end{align*}
Then $\wlam$ can be written as the orthogonal sum $\bigoplus_{k=1}^\infty J_k$, where $J_k$ is a truncated weighted shift with weights $\{a_k,a_k\}$ ($k\in\nbb$). In view of Corollary \ref{mymister}, $\wlam$ is complex selfadjoint (independently of the choice of the sequence $(a_n)_{n=1}^\infty$). As shown in Lemma \ref{terapia}, $\wlam^m$ is complex symmetric for every $m\in\nbb$.

We now show that one can choose $(a_n)_{n=1}^\infty$ so as to get non closed (in particular, non complex selfadjoint) $\wlam^2$. We begin by picking $(b_n)_{n=1}^\infty$ that satisfies 
\begin{align}\label{drama}
\sum_{k=1}^\infty |b_k|^2<\infty   \quad \text{ and }\quad b_k\neq 0, k\in\nbb,
\end{align}
and then we choose $(a_n)_{n=1}^\infty$ such that
\begin{align}\label{fotel}
\sum_{k=1}^\infty |a_kb_k|^2=\infty.
\end{align}
Clearly, for such a sequence $(a_n)_{n=1}^\infty$, $\wlam$ is not bounded. Also, it is not closed. Indeed, let $f\colon \nbb\to\cbb$ be given by
\begin{align*}
f(k)=\begin{cases}
0 & \text{ for } n=3k\text{ or }n=3k-2,\\
b_k & \text{ for }n=3k-1,
\end{cases}\quad (k\in\nbb).
\end{align*}
Then $f\in\ell^2(\nbb)$ because of \eqref{drama}. For any $n\in\nbb$, let $f^{(n)}\colon \nbb\to\cbb$ be defined by $f^{(n)}=\chi_{\{1,\ldots,n\}} f$, where $\chi_{\{1,\ldots,n\}}$ is the characteristic function of the set $\{1,\ldots,n\}$. It is easily seen that $(f^{(n)})_{n=1}^\infty\subseteq \lin{\ee}\subseteq\dz{\wlam^2}$ and $f=\lim_{n\to\infty} f^{(n)}$. Moreover, by \eqref{kdrama}, we have $\wlam^2 f^{(n)}=0$ for every $n\in\nbb$. On the other hand, $f\notin\dz{\wlam^2}$ since \eqref{fotel} implies that $f\notin\dz{D_{\lambdab}}$. Therefore $\wlam^2$ cannot be closed. The non closedness of $\wlam^2$ can also be shown using more general results from theory of weighted composition operators (see Remark \ref{prom} below). Clearly, $\wlam^2$ cannot be $C$-selfadjoint as well and $(\wlam^{*})^2\neq (\wlam^{2})^*$.

It is worth pointing out that the conjugation $C$ with respect to whom $\wlam$ is complex symmetric can be written as the orthogonal sum $\bigoplus_{n=1}^\infty C_n$ with every $C_n$ acting on $\cbb^3$. Consequently, we have $C\ee\subseteq\lin{\ee}\subseteq \dzn{\wlam}$.
\end{exa}
\begin{rem}\label{prom}
Considering the question of closedness of powers of $\wlam$ one can show, using results on weighted composition operators \cite[Sections 2.3 (g) and 4.1]{2018-bud-jab-jun-sto-lnm}, that $\wlam^2$ is closed if and only if there exists a positive $c$ such that
\begin{align*}
|\lambda_k|^2\leqslant c\big(1+|\lambda_k\lambda_{k+1}|^2\big),\quad k\in\mathbb{N}
\end{align*}
holds.
\end{rem}
It turns out that $C$-symmetric weighted shift is necessarily $C$-selfadjoint whenever the conjugation satisfies additional condition 
\begin{align}\label{simba}
C\mathcal{E}\subseteq\dz{\wlam}.
\end{align}
This follows from a more general result that is given below. Let us recall that a complex symmetric operator is closable.
\begin{pro}\label{lucek}
Let $\ff$ be a linear subspace of a Hilbert space $\hh$. Let $T$ be an operator in $\hh$ which is complex symmetric with respect to a conjugation $C$ satisfying $C\ff\subseteq \dz{\overline T}$. Assume that $\ff$ is a core for $T^*$. Then $\overline T$ is $C$-selfadjoint.
\end{pro}
\begin{proof}
Since $\ff\subseteq\dz{T^*}$ and $T^*=\overline{T}^*$, we have
\begin{align}\label{corona1}
C\overline{T} C|_{\ff}=T^*|_{\ff}.
\end{align}
Indeed, since $T$ is $C$-symmetric, $\overline{T}$ is $C$-symmetric and thus $\overline{T} g=CT^*Cg$ for every $g\in\dz{\overline{T}}$. Substituting $g=Cf$ with $f\in\ff$ we get \eqref{corona1}.

Since $\ff$ is a core for $T^*$ and $C\overline{T}C$ is closed, we get
\begin{align*}
T^*=\overline{T^*|_{\ff}}=\overline{C\overline{T} C|_{\ff}}\subseteq\overline{C\overline{T} C}=C\overline{T} C.
\end{align*}
This completes the proof.
\end{proof}
The above yields the following.
\begin{cor}\label{coronavirus}
Let $\wlam$ be $C$-symmetric. Then $\wlam$ is $C$-selfadjoint if and only if \eqref{simba} holds.
\end{cor}
\begin{proof}
It is known that $\lin\mathcal{E}$ is a core for the adjoint of a weighted shift (see \cite[Eq. (1.11)]{mlak-uiam-1988}), so Proposition \ref{lucek} implies the ``if'' part. The ``only if'' part is obvious.
\end{proof}
\begin{rem}\label{inka}
As shown in Example \ref{dostawa}, it is in general not true that the power of a weighted shift is a closed operator even if the weighted shift is complex selfadjoint. This means that in general one cannot apply Corollary \ref{coronavirus} without assuming closedness to the powers of a given weighted shift or passing to the closure of the operator in question.
\end{rem}
\begin{cor}\label{kalosze}
Let $\wlam$ be $C$-symmetric and let $k\in \nbb$. Then the following hold:
\begin{itemize}
\item[(i)] Assume that $\wlam^k$ is closed. Then $\wlam^k$ is $C$-selfadjoint if and only if $C \mathcal{E}\subseteq \dz{\wlam^k}$.
\item[(ii)] $\overline{\wlam^k}$ is $C$-selfadjoint if and only if $C \mathcal{E}\subseteq \dz{\overline{\wlam^k}}$
\end{itemize}
\end{cor}
%
The following is an unbounded weighted shift counterpart of Theorem \ref{resolvent}.
\begin{pro}\label{beck+}
Let $\wlam$ be $C$-symmetric with $C$ satisfying \eqref{simba}. Suppose $\wlam=U|\wlam|$ is the polar decomposition of $\wlam$. Then there exists a partial conjugation $J$ such that $U=C J$ and $J|\wlam|=|\wlam|J$.
\end{pro}
\begin{proof}
By Corollary \ref{coronavirus}, $\wlam$ is $C$-selfadjoint and thus the claim follows from Theorem \ref{beck-}.
\end{proof}
\begin{rem}\label{jakzwykle}
An analogous result to the above one can be stated for the powers of $\wlam$: assuming that $\wlam$ is $C$-symmetric with $C\ee\subseteq\dz{\overline{\wlam^k}}$, there is a partial conjugation $J_k$ that commutes with $|\overline{\wlam^k}|$ and the partial isometry $U_k$ in the polar decomposition $\overline{\wlam^k}=U_k|\overline{\wlam^k}|$ is the product of $C$ and $J_k$. Since $U_k=U^k$ with $U$ being the partial isometry in the polar decomposition of $\wlam$ (one can prove this using Remark \ref{cacit}), we see that $J_k=U^{k*}C$.
\end{rem}
It is well known that $C\nul{T}=\nul{T^*}$ for any $C$-selfadjoint $T$. Since any weighted shift $\wlam$ with non-zero weights is necessarily injective and the kernel of its adjoint is non trivial, Corollary \ref{coronavirus} yields
\begin{cor}\label{tajfun}
There is no weighted shift $\wlam$ with nonzero weights that is complex symmetric with respect to a conjugation $C$ satisfying \eqref{simba}.
\end{cor}
This leads to the following interesting problem.
\begin{opq}\label{wtn}
Does there exist a unilateral weighted shift $\wlam$ in $\ell^2(\nbb)$ with positive weights $\lambdab$ that is complex symmetric?
\end{opq}
In the remark below we give a reformulation of the Problem \ref{wtn} as a question concerning sequences of positive numbers.
\begin{rem}
Suppose $\wlam$ is a weighted unilateral shift operator with positive weights that is complex symmetric with respect to a conjugation $C$. Clearly, it satisfies 
\begin{align}\label{q4}
\wlam e_k=C\wlam^* C e_k,\quad k\in\nbb.
\end{align}
In view of Corollary \ref{tajfun}, $C\ee\not\subseteq\dz{\wlam}$. Assuming for simplicity that 
\begin{align}\label{portki1}
C e_1=\sum_{k=1}^\infty \alpha_k e_k\not\in \dz{\wlam}   
\end{align}
with $\{\alpha_n\}_{n=1}^\infty\subseteq (0,+\infty)$, we deduce two following conditions:
\begin{align}\label{spodnie1}
\sum_{k=1}^\infty \alpha_k^2=1 \quad \text{and}\quad  \sum_{k=1}^\infty \alpha_k^2\lambda_{k}^2=\infty
\end{align}
Applying repeatedly \eqref{kdrama}, \eqref{q4} and \eqref{portki1} we get
\begin{align}\label{portki2}
C e_{1+m}=\sum_{k=1}^\infty \frac{\lambda_{k}^2\cdots \lambda_{k+m-1}^2}{\lambda_1^2\cdots \lambda_m^2}\alpha_{k+m}^2 e_{k}, \quad m\in\nbb.
\end{align}
This implies
\begin{align}\label{spodnie3}
\sum_{k=1}^\infty {\lambda_{k}^2\cdots \lambda_{k+m-1}^2}\alpha_{k+m}^2 ={\lambda_1^2\cdots \lambda_m^2}, \quad m\in\nbb.
\end{align}
It is clear that \eqref{portki1} and \eqref{portki2} can be used as defining formulas for $C$. Thus, finding sequences 
$\{\alpha_n\}_{n=1}^\infty$ and $\{\lambda_{n}\}_{n=1}^\infty$ of positive real numbers satisfying \eqref{spodnie1} and \eqref{spodnie3} yields a positive solution to Problem \ref{wtn}.
\end{rem}
Clearly, for any $\wlam$ that is $C$-symmetric with $C$ satisfying \eqref{simba}, the set $$\Lambda_0:=\{i\in \nbb \colon \lambda_i=0\}$$ is not empty. Moreover, as observed below, it cannot be finite. 
\begin{pro}\label{balans}
Let $\wlam$ be $C$-symmetric with $C$ satisfying \eqref{simba}. Then $\Lambda_0$ is infinite. 
\end{pro}
\begin{proof}
Suppose $\Lambda_0$ is finite. $\wlam$ has to be $C$-selfadjoint by Corollary \ref{coronavirus} and so the dimensions of $\nul{\wlam}$ and $\nul{\wlam^*}$ are equal. Since $\wlam$ can be decomposed into the orthogonal sum of a truncated weighted shift and a unilateral weighted shift with nonzero weights (see Remark \ref{cash}) we see that the dimension of the kernel of $\wlam^*$ turns out to be greater than the dimension of the kernel of $\wlam$. Hence we obtain a contradiction and so $\Lambda_0$ cannot be finite.
\end{proof}
The above yields yet another problem.
\begin{opq}
Does there exist a complex symmetric weighted shift with $\Lambda_0$ being finite?
\end{opq}
Assuming that $\wlam$ is complex symmetric and $\Lambda_0$ is finite, there exists $k\in\nbb$ such that $\overline{\wlam^k}$ can be decomposed into the orthogonal sum of zero operator acting on a finite dimensional spaces and the orthogonal sum $W_{\lambdab_1}\oplus \ldots\oplus W_{\lambdab_k}$ of unilateral weighted shifts with nonzero weights (see Remark \ref{cacit}). By Lemma \ref{terapia}, $\wlam^k$ is complex symmetric and if closed, $W_{\lambdab_1}\oplus \ldots\oplus W_{\lambdab_k}$ is complex symmetric as well. Hence the positive solution to Problem \ref{wtn} gives a weighted shift with finite $\Lambda_0$ such that the closure of its $k$th power is complex selfadjoint.

The following lemma is an unbounded counterpart of \cite[Proposition 3.2]{z-l-tams-2013}. The general idea behind the proof is as of the one used in proving the bounded version but is technically  much more subtle and also, by preference not necessity, uses the partial conjugation $J$ given by Proposition \ref{beck+}.
\begin{lem}\label{lucek+}
Let $n\in\mathbb{N}$. Let $\wlam=\bigoplus_{i\in\mathcal{I}}W_i$, where each $W_i$ is a truncated weighted shift of order $n\in\nbb$. Suppose $\wlam$ is complex symmetric with respect to a conjugation $C$ satisfying $C\ee\subseteq\dzn{\wlam}$. Then $\wlam$ can be written as an orthogonal sum 
$\bigoplus_{l\in\mathcal{L}} J_{\mub^{(l)}}$, where each $J_{\mub^{(l)}}$, $l\in\mathcal{L}$, is a truncated weighted shift on either $\cbb^{n+1}$ or $\cbb^{2n+2}$ with weights ${\mub^{(i)}}$ satisfying either $|\mu_k^{(l)}|=|\mu_{n+1-k}^{(l)}|$ for $k=1, \ldots, n$ or $|\mu_k^{(l)}|=|\mu_{2n+2-k}^{(l)}|$ for $k=1, \ldots, 2n+1$
\end{lem}
\begin{proof}
By Corollary \ref{coronavirus}, $\wlam$ is $C$-selfadjoint. For $k\in\nbb$, let $\wlamk$ be an operator defined in Remark \ref{cacit}. Recall that $\wlamk=\overline{\wlam^k}$ (and $W_{\lambdab,1}=\wlam$) and $\wlamk$ is unitarily equivalent to $W_{\lambdab_1}\oplus \ldots\oplus W_{\lambdab_k}$ with $\lambdab_i$'s defined via \eqref{poezje}. By Corollary \ref{kalosze} every $\wlamk$ is $C$-selfadjoint. In view of Proposition \ref{beck+} (see also Remark \ref{jakzwykle}), for every $k\in\nbb$ we can write $\wlamk=CJ_k|\wlamk|$ with $J_k=U^{k*}C$ commuting with $|\wlamk|$ (here $J_1:=J)$, where $U$ is taken from the polar decomposition $\wlam=U|\wlam|$. Thus we see that 
\begin{align*}
\text{$J_k=J (CJ)^{k-1}$ for all natural $k\geqslant 2$.}    
\end{align*}
Note that 
\begin{align*}
\nul{J_k}=\nul{\wlamk}=\ee^1\oplus\ldots\oplus\ee^k
\end{align*}
and
\begin{align*}
\nul{\wlamk^*}=\ee^{n+1-k}\oplus\ldots\oplus\ee^{n+1}    
\end{align*}
with $\ee^l=\bigvee\{e_l^{i}\colon i\in\nbb\}$. We now show that
\begin{align}\label{dyskJ}
\text{$J\ee^i=\ee^{n-i+3}$ for $i=2,\ldots, n+1$}
\end{align}
and
\begin{align}\label{dyskC}
\text{$C\ee^i=\ee^{n-i+2}$ for $i=1,\ldots, n+1$.}
\end{align}
Clearly, $C\nul{\wlamk}=\nul{\wlamk^*}=\nul{\wlam^{*k}}$ for every $k=1,\ldots n$ which implies \eqref{dyskC}. For the proof of \eqref{dyskJ}, we fix $i=2,\ldots, n+1$ and observe that $J\ee^i=\ee^{n-i+3}$ if and only if $CJ\ee^i=C\ee^{n-i+3}$. The latter clearly holds by \eqref{dyskC}.

Observe that we have
\begin{align*}
Je^i_{k+2}= \alpha(k,i) J \wlamk e^i_2&= \alpha(k,i) U^*C\wlamk e^i_2\notag\\
&=\alpha(k,i) U^*\wlamk^*Ce^i_2,\quad i\in\mathcal{I}, k=0, 1, \ldots, n-1,
\end{align*}
with some nonzero constant $\alpha(k,i)\in\mathbb{C}$. Suppose now that $\beta_i\colon \mathcal I \to\mathbb{C}$ is a function such that
\begin{align*}
Ce^i_2 =\sum_{j} \beta_i(j)e^j_n.
\end{align*}
Then, in view of \eqref{dyskC}, for every fixed $i\in\mathcal{I}$, the support of $\beta_i$ is the same as the support of a function $\gamma_{(i,k)}$, $k=0, \dots, n-1$, such that (cf. \eqref{dyskJ})
\begin{align*}
Je^i_{k+2}= \sum_{j} \gamma_{(i,k)}(j)e^j_{n-k+1}.
\end{align*}
Denote it by $\Gamma_i$, $i\in\mathcal{I}$. Since $|\wlam|=\bigoplus_{i\in\mathcal{I}} D_{\tilde\lambdab^{(i)}}$ with $\tilde\lambdab^{(i)}=\big(0,|\lambda_1^{(i)}|, \ldots, |\lambda_n^{(i)}|\big)$ and $J|\wlam|=|\wlam|J$ we deduce
\begin{align}\label{spiulki}
|\lambda_1^{(i)}|=|\lambda_n^{(j)}|, |\lambda_2^{(i)}|=|\lambda_{n-1}^{(j)}|, \ldots, |\lambda_n^{i}|=|\lambda_{1}^{j}|,\quad j\in\Gamma_i, i\in\mathcal{I}.
\end{align}
Now, if $i\in\Gamma_i$, then 
\begin{align}\label{martin1}
|\lambda^{(i)}_k|=|\lambda^{(i)}_{n+1-k}|,\quad k=1, \ldots, n.
\end{align}
For other $j$'s in $\Gamma_j$, we define 
\begin{align*}
\widehat \lambdab^{(i,j)}=\big(\lambda_1^{i}, \lambda_2^{i},\ldots \lambda_n^{i}, 0, \lambda_1^{j}, \lambda_2^{j},\ldots \lambda_n^{j}\big)\in\mathbb{C}^{2n+1}.
\end{align*}
By \eqref{spiulki}, we clearly have
\begin{align}\label{martin2}
|\widehat\lambda^{(i,j)}_k|=|\widehat\lambda^{(i,j)}_{2n+2-k}|,\quad k=1, \ldots, 2n+1.
\end{align}
Since there are as many $(|\lambda_1|, \ldots, |\lambda_n|)$ in $\{|\lambdab_i|\colon i\in \mathcal{I}\}$ as $(|\lambda_n|, \ldots, |\lambda_1|)$ (see the argument in the last paragraph in \cite[pg. 517]{z-l-tams-2013}), we can rearrange $\bigoplus_{i\in\mathcal{I}}W_i$ into $\big(\bigoplus_{j\in\mathcal{J}}W_j\big) \oplus \big(\bigoplus_{k\in\mathcal{K}}W_k\big)$ so as the weights $\lambdab^{(j)}$ of $W_j$, for $j\in\mathcal{J}$, satisfy \eqref{martin1} and the weights $\widehat\lambdab^{(2k-1, 2k)}$ of $W_{2k-1}\oplus W_{2k}$, for $k\in\mathcal{K}$, satisfy \eqref{martin2}. This yields the claim.
\end{proof}
The next lemma is extracted from the proof of Theorem \ref{ul}.
\begin{lem}\label{debile}
Let $k\in\nbb$. Let $\hh$ and $\kk$ be Hilbert spaces and $A\oplus B$ be a closed operator in $\hh\oplus\kk$ such that for $1\leq m\leq k$, $A^m$ and $B^m$ are closable, and $\overline{A^m}\oplus \overline{B^m}$ is complex selfadjoint. Let $\mathcal{F}$ be a linearly dense subspace of $\hh$. Assume that $\mathcal{F}\subseteq \bigcup_{1\leq m\leq k} \jd{\overline{A^m}}\cap\jd{A^{(k+1-m)*}}$ and $\bigcup_{1\leq m\leq k}\jd{B^{m*}}\cap\jd{\overline{B^{k+1-m}}}=\{0\}$. Then $A$ and $B$ are complex selfadjoint.
\end{lem}
\begin{proof}
Let $f\in\mathcal{F}$. Then there is $1\leq n\leq k$ such that $f\in \jd{\overline{A^n}}\cap \jd{A^{l*}}$ with $l=k+1-n$. Let $C$ denote the conjugation with respect to which every $\overline{A^m}\oplus \overline{B^m}$, $1\leq m\leq k$, is complex selfadjoint  (inspecting the proof of Lemma \ref{terapia} one can see that for all $m$'s the same conjugation is good). Clearly, $A^{l*}\oplus B^{l*}$ is $C$-selfadjoint as well. Therefore, $Cf\in\jd{(A^m\oplus B^m)^*}=\jd{A^{m*}}\oplus \jd{B^{m*}}$ and $Cf\in\jd{(A^l\oplus B^l)^{**}}=\jd{A^{l**}}\oplus \jd{B^{l**}}=\jd{\overline{A^{l}}}\oplus \jd{\overline{B^{l}}}$, with $f$ isometrically embedded in $\hh\oplus \kk$ (the last equality holds by von Neumann theorem and closability of $A^l$ and $B^l$). This implies $Cf=g+h$ with $g\in\hh$ and $h\in\jd{B^{m*}}\cap\jd{\overline{B^{l}}}$. By our assumptions $h=0$. Since $f$ can be chosen arbitrarily, we have $C\hh\subseteq \hh$. This obviously yields $C\hh=\hh$ and consequently $C\kk=\kk$. Therefore $C=C|_\hh\oplus C|_\kk$ and so the claim follows.
\end{proof}
\begin{rem}\label{debile+}
It is easy to see that $A$ being equal to an orthogonal sum of truncated weighted shifts of order $k-1$ and $B$ being equal to an orthogonal sum of truncated weighted shifts of order greater or equal to $k$ satisfy all the assumptions of the above lemma with $\ff=\ee$.
\end{rem}
\begin{thm}\label{ul}
If $\wlam$ is $C$-symmetric with $C$ satisfying $C \mathcal{E}\subseteq \dzn{\wlam}$, then $\wlam$ is $C$-selfadjoint. Moreover, $\wlam$ can be decomposed into the orthogonal sum 
$\wlam=\bigoplus_{i\in\mathcal{I}} J_{\lambdab^{(i)}}$, where each $J_{\lambdab^{(i)}}$, $i\in\mathcal{I}$, is a truncated weighted shift on $\cbb^{n_{i}}$ with weights ${\lambdab^{(i)}}$ satisfying $|\lambda_k^{(i)}|=|\lambda_{n_{i}-k}^{(i)}|$ for $k=1, \ldots, n_{i}-1$. 
\end{thm}
\begin{proof}
Since $C \mathcal{E}\subseteq \dzn{\wlam}$ implies $C \mathcal{E}\subseteq \dz{\wlam}$, Corollary \ref{coronavirus} yields the $C$-selfadjointness of $\wlam$. By Corollary \ref{balans}, $\Lambda_0=\{i: \lambda_i=0\}$ is infinite. 

Since the set $\Lambda_0$ is infinite $\wlam$ can be decomposed into the orthogonal sum 
\begin{align}\label{ketonal+}
\wlam=\bigoplus_{i\in\mathcal{I}} J_{\lambdab^{(i)}},
\end{align}
where each $J_{\lambdab^{(i)}}$, $i\in\mathcal{I}$, is a truncated weighted shift with nonzero weights (see Remark \ref{cash}). We may rearrange \eqref{ketonal+} so as to have
\begin{align*}
\wlam=\bigoplus_{j\in\mathcal{J}} W_j,
\end{align*}
where each $W_j$ is the orthogonal sum of all $J_{\lambdab^{(i)}}$'s that are of rank $j$. Let $\hh=\bigoplus_{j\in\mathcal{J}} \hh_j$ the underlying decomposition of $\hh$. 

Using Lemma \ref{debile} (see also Remark \ref{debile+}) we deduce that every $\hh_j$ is reducing $C$ and $W_j$ is $C|_{\hh_j}$-selfadjoint. Now it suffices to use Lemma \ref{lucek+} to complete the proof.
\end{proof}
It is obvious that dealing with the $C^\infty$ vectors of any given operator is more difficult that with the domain. Hence the following arises.
\begin{opq}
Suppose $\wlam$ is $C$-symmetric and $C\ee\subseteq\dz{\wlam}$. Can $\wlam$ be decomposed as in Theorem \ref{ul}?
\end{opq}
\section{The bilateral case}
In the finishing section of the paper we discuss briefly complex symmetry of unbounded bilateral weighted shifts. Zhu and Li proved characterization of their $C$-symmetry in the bounded case in \cite[Theorem 4.1]{z-l-tams-2013}. As shown below in Theorem \ref{vw01}, this characterization remains valid in the unbounded case provided we assume, analogously to the unilateral case, that the set of $C^\infty$ vectors of $\wlam$ contains $C\ee$, where $\ee$ is the canonical orthonormal basis of $\ell^2(\zbb)$. 

Let us recall that the {\em bilateral weighted shift} with weights $\lambdab=\{\lambda_n\}_{n={-\infty}}^\infty\subseteq \cbb$ is defined by
\begin{align*}
\dz{\wlam} &= \big\{f\in \ell^2(\zbb)\colon \sum_{n=-\infty}^{\infty} |f(n)\lambda_{n}|^2 <\infty\big\},\\
(\wlam f)(n)&=\lambda_{n-1} f(n-1),\quad n\in\zbb,\ f\in\dz{\wlam}.
\end{align*}
It is known that that $\wlam e_n=\lambda_ne_{n+1}$, $\wlam^*e_n=\lambda_{n-1}e_{n-1}$, and $|\wlam|e_n=|\lambda_n|e_n$ for $n\in\zbb$.

The general idea of the proof is the same as in the bounded case, however we use primarily the partial conjugation $J$ (see Theorem \ref{beck-}) instead of the conjugation $C$ as it seems more convenient.
\begin{thm}\label{vw01}
Let $\wlam$ be a bilateral weighted shift. Assume that $C$ is a conjugation on $\ell^2(\zbb)$ such that $C\ee\subseteq \dzn{\wlam}$. Then $\wlam$ is $C$-symmetric 
 if and only if one of the following conditions holds:
\begin{itemize}
    \item[(i)] $\{\lambda_n\}_{n=-\infty}^\infty\subseteq \cbb\setminus\{0\}$ and there exists $k\in\zbb$ such that $|\lambda_n|=|\lambda_{k-n}|$ for all $n\in\zbb$,
    \item[(ii)] $\wlam$ is unitarily equivalent to $W_{\gammab}^*\oplus W_{\gammab}$, where $W_{\gammab}$ is an injective unilateral weighted shift on $\ell^2(\nbb)$,
    \item[(iii)] $\wlam$ is unitarily equivalent to $W_{\gammab}^*\oplus W_{\gammab}\oplus\, \bigoplus_{i=1}^k J_{\lambda^{(i)}}$, where $W_{\gammab}$ is an injective unilateral weighted shift on $\ell^2(\nbb)$ and each $J_{\lambda^{(i)}}$ is a complex selfadjoint truncated weighted shift (cf. Lemma \ref{lucek+}),
    \item[(iv)]$\wlam$ is unitarily equivalent to a complex symmetric unilateral weighted shift on $\ell^2(\nbb)$ (cf. Theorem \ref{ul}).
\end{itemize}
\end{thm}
\begin{proof}
Assume that $\wlam$ is $C$-symmetric. Since $C\ee\subseteq \dzn{\wlam}$, $\wlam$ is $C$-selfadjoint  (use Proposition \ref{lucek} and the fact that $\lin\ee$ is a core for $\wlam^*$). Let $J$ be the partial conjugation given by Theorem \ref{beck-} and $U$ be the partial isometry from the polar decomposition $\wlam=U|\wlam|$. Le us consider the following four cases.

{\em Case 1.} For every $n\in\zbb$, $\lambda_n\neq 0$.

Since $J$ commutes with $|\wlam|$, we see that
\begin{align}\label{riot01}
Je_n\in \bigvee\{e_i\colon |\lambda_i|=|\lambda_n|\},\quad n\in\zbb.
\end{align}
$\wlam$ is injective and thus $\nul{J}=\nul{\wlam}=\{0\}$. Hence, $\is{Je_0}{e_k}\neq 0$ for some $k\in\zbb$. This implies that 
\begin{align}\label{riot02}
\is{Je_n}{e_{k-n}}\neq 0, \quad n\in\zbb.
\end{align}
This can be proved in the following way. We begin with $n=1$. Since $\wlam$ is $C$-selfadjoint, we have
\begin{align*}
U^*\wlam^*U=U^*\wlam^*CJ=U^*C\wlam J=J\wlam J.
\end{align*}
$U^*\wlam^*Ue_1=\bar\lambda_1 e_0$ and so we have
\begin{align*}
0\neq\is{J\wlam Je_1}{e_0}=\is{Je_1}{\wlam^*Je_0}.
\end{align*}
$\is{Je_0}{e_k}\neq 0$ implies $Je_0=\alpha e_k+f$, where $\alpha\neq0$ and $f\perp e_k$. Thus $\wlam^*Je_0=\alpha \bar \lambda_{k-1} e_{k-1} +\wlam^*f$. Hence we obtain $\is{Je_1}{e_{k-1}}\neq 0$. For the proof of \eqref{riot02} with $n$ other than $0$ and $1$, we use a recursive argument. Combining \eqref{riot01} and \eqref{riot02} we that $|\lambda_n|=|\lambda_{k-n}|$ for all $n\in\zbb$, meaning that (i) holds.

{\em Case 2.} There exists a unique $k\in\zbb$ such that $\lambda_k=0$.

Without loss of generality we may assume that $k=0$. Since $J|\wlam|=|\wlam| J$, we get $Je_0=\alpha_0 e_0$ with $|\alpha_0|=1$. This implies
\begin{align}\label{riot03}
Je_{-k}=\alpha_k e_k, \text{ with $\alpha_k=1$},\quad k\in \zbb.
\end{align}
Let us prove the above with $k=1$. By the $C$-sefladjointness of $\wlam$ we have 
\begin{align}\label{riot04}
J\wlam = U^*C \wlam = U^* \wlam^* C= U^* (U|\wlam|)^* C=U^*|\wlam|U^* C=U^*|\wlam|J.
\end{align}
This implies 
\begin{align*}
\bar \lambda_n J e_{n+1}=|\lambda_n| U^* J e_n,\quad n\in\zbb.
\end{align*}
Substituting $n=-1$ into the above, we get
\begin{align*}
\bar\lambda_{-1} \alpha_0= |\lambda_{-1}| U^*J e_{-1},
\end{align*}
which implies the equality in \eqref{riot03} with $k=-1$. For the proof of \eqref{riot03} with other $k$'s, we use a recursive argument. In view of \eqref{riot01} (holding for all $n$'s and $i$'s with the respective weights nonzero) we see that
\begin{align}\label{riot05}
|\lambda_{-k}|=|\lambda_k|,\quad k\in\zbb.
\end{align}
Since $\lambda_0=0$, $\wlam$ is unitarily equivalent the orthogonal sum $W_{\lambdab^{-}}^*\oplus W_{\lambdab^{+}}$ where $W_{\lambdab^{-}}$ and $W_{\lambdab^{+}}$ are unilateral weighted shifts on $\ell^2(\nbb)$ with weights $\lambdab^{-}=\{\lambda_{-n}\}_{n=1}^\infty$ and $\lambdab^{+}=\{\lambda_{n}\}_{n=1}^\infty$, respectively. In view of \eqref{riot05}, we see that $\wlam$ is in fact unitarily equivalent the orthogonal sum $W_{\lambdab^{+}}^*\oplus W_{\lambdab^{+}}$. This leads to (ii).

{\em Case 3.} There exist $k_1<k_2<\ldots <k_N$ with $N\in \nbb\setminus \{1\}$ and $k_i\in\zbb$ such that $\lambda_{k_j}=0$ for $j=1, \ldots, N$ and $\lambda_k\neq 0$ for all $k\in\zbb\setminus \{k_1, \ldots, k_N\}$.

Under the hypothesis, $\wlam$ is unitarily equivalent to the orthogonal sum $W_{\lambdab^-}^*\oplus J_{\widehat \lambdab}\oplus W_{\lambdab^+}$, where $W_{\lambdab^-}$ and $W_{\lambdab^+}$ are unilateral weighted shifts on $\ell^2(\nbb)$ with weights $\lambdab^-=\{\lambda_{k_1-n}\}_{n=1}^\infty$ and $\lambdab^+=\{\lambda_{k_N+n}\}_{n=1}^\infty$, respectively, and $J_{\widehat\lambdab}$ is a truncated unilateral weighted shift with weights $\widehat\lambdab=\{\lambda_{k_1+1},\ldots ,\lambda_{k_N-1}\}$. Using Lemma \ref{debile} we get that $J_{\widehat\lambdab}$ and $W_{\lambdab^-}^*\oplus W_{\lambdab^+}$ are complex selfadjoint. Clearly, $J_{\widehat\lambdab}$ is unitarily eqivalent to $\bigoplus_{i=1}^k J_{\lambda^{(i)}}$ and all $J_{\lambda^{(i)}}$'s are complex selfadjoint (for this one can use the original result due to Zhu and Li since $J_{\widehat\lambdab}$ is bounded). On the other hand, $W_{\lambdab^-}^*\oplus W_{\lambdab^+}$ is unitarily equivalent to a bilateral weighted shift $W_{\widetilde\lambdab}$ with weights $\widetilde\lambdab = \{\ldots,  \lambda_{k_1+2}, \lambda_{k_1-1}, 0, \lambda_{k_N+1}, \lambda_{k_N+2}, \ldots \}$. Thus, by the complex selfadjointness of $W_{\lambdab^-}^*\oplus W_{\lambdab^+}$ and {\em Case 2}, $|\lambda_{k_1-n}|=|\lambda_{k_N+n}|$ for all $n\in\nbb$. Therefore, $W_{\lambdab^-}^*\oplus W_{\lambdab^+}$ is unitarily eqivalent to $W_{\gammab}^*\oplus W_{\gammab}$ with $\gammab=\{|\lambda_{k_N+n}|\}_{n=1}^\infty$.

{\em Case 4.} The set $\{n\in\zbb\colon \lambda_n=0\}$ is infinite.

The set $\{n\in\zbb\colon \lambda_n=0\}$ has neither an upper nor a lower bound. Indeed, suppose that one of the bounds exists. Then $\ell^2(\zbb)$ can be decomposed into the orthogonal sum $\hh\oplus \kk$ such that $\wlam$ acts on $\hh$ as either $W_{\lambdab^{(0)}}$ or $W_{\lambdab^{(0)}}^*$, where $W_{\lambdab^{(0)}}$ is an injective unilateral weighted shift, and $\wlam$ acts on $\kk$ as $\bigoplus_{i=1}^\infty J_{\lambdab^{(i)}}$, where each $J_{\lambdab^{(i)}}$ is a truncated weighted shift (cf. Remark \ref{cash}). Since for every $k\in\nbb$, $C\nul{W_{\lambdab,k}}=\nul{\wlam^{*k}}$ (we define $W_{\lambdab,k}$ in $\ell^2(\zbb)$ in a similar way as in $\ell^2(\nbb)$ case; see Remark \ref{cacit}) one can deduce that $C\hh\oplus \kk\subseteq \kk$, which clearly is a contradiction.

Since $\{n\in\zbb\colon \lambda_n=0\}$ has no bounds, the summand $W_{\lambdab^{(0)}}$ does not appear in the decomposition and so $\wlam$ is unitarily equivalent (by rearranging the basis) to $\bigoplus_{i=1}^\infty J_{\lambdab^{(i)}}$, which in turn is unitarily equivalent to a unilateral weighted shift on $\ell^2(\nbb)$. This gives (iv).

Thse sufficiency of (i)-(iv) can be proved in a same fashion as in the bounded case (we leave the details to the reader).
\end{proof}
\begin{rem}
It is worth noticing that the arguments in {\em Cases 1 $\&$ 2} do not require $C\ee\subseteq \dzn{\wlam}$ to be assumed. Assuming less, namely $C\ee\subseteq \dz{\wlam}$ is enough. Indeed, this assumption and $C$-symmetry imply together that $J$ exists, $J\ee\subseteq \dz{\wlam}$, $J\wlam J e_n=U^*\wlam^*Ue_n$, and $J\wlam e_n=U^*|\wlam|Je_n$, which is enough to proceed in the same fashion.
\end{rem}
\begin{rem}
Inspecting the proof of {\em Case 1} one can see that the condition ``$|\lambda_n|=|\lambda_{k-n}|$ for all $n\in\zbb$'' is satisfied with every $k\in\zbb$ such that $\is{Je_0}{e_k}\neq 0$.  

Suppose that there are two distinct $k$'s such that the condition holds. Denote them $k_1$ and $k_2$ and assume (without loss of generality) that $k_1<k_2$. For $i\in\zbb$, let $\Delta_i=\{n\colon |\lambda_i|=|\lambda_n|\}$. Then one can deduce that $\bigcup_{i\in\{k_1,k_1+1,\ldots ,k_2\}}\Delta_i=\zbb$. In particular, this means that the set $\{|\lambda_n|\colon n\in\zbb\}$ is in fact finite. Hence, if $\wlam$ is $C$-symmetric, $J\ee\subseteq\dz{\wlam}$, and ``$\is{Je_0}{e_k}\neq 0$ for all $n\in\zbb$'' holds for at least two $k$'s, $\wlam$ has to be a bounded operator on $\ell^2(\zbb)$.

On the other hand, if there is a single $k\in\zbb$ such that $\is{Je_0}{e_k}\neq 0$ holds, we have $Je_n=\alpha_n e_{k-n}$ for all $n\in\zbb$, with $|\alpha_n|=1$.
\end{rem}
\section*{Acknowledgments}

The first author was supported in part by the Labex CEMPI  (ANR-11-LABX-0007-01). The second author was supported by the Ministry of Science and Higher Education of the Republic of Poland.

\end{document}